\documentclass[11pt]{amsart}
\usepackage{amsmath}
\usepackage{amsfonts}
\usepackage{amssymb}
\usepackage{amsthm}
\usepackage{mathrsfs}       
\usepackage{graphicx}     
\usepackage[utf8]{inputenc}
\allowdisplaybreaks

\usepackage{xspace}
\usepackage[colorlinks=true]{hyperref}

\newcommand{\R}{\mathbb{R}}
\newcommand{\N}{\mathbb{N}}

\newcommand{\C}{\mathscr{C}}
\newcommand{\G}{\mathscr{G}}

\newcommand{\ud}{\,\mathrm{d}}

\newtheorem{thm}{THEOREM}[section]
\newtheorem{remark}[thm]{REMARK}
\newtheorem{lemma}[thm]{LEMMA}
\newtheorem{definition}[thm]{DEFINITION}
\newtheorem{proposition}[thm]{PROPOSITION}
\newtheorem{corollary}[thm]{COROLLARY}

\newcounter{thmbiss}

\title{Flatness for a strongly degenerate 1-D parabolic equation}

 \author[I.~Moyano]{Iv\'an Moyano}
\begin{address}{Iv\'an Moyano,
Centre de math\'ematiques Laurent Schwartz, UMR 7640, Ecole polytechnique, Palaiseau, France}
  \email{\href{mailto:ivan.moyano@math.polytechnique.fr}{ivan.moyano@math.polytechnique.fr}}
 \end{address}

\begin{document}
\maketitle

\begin{abstract}
We consider the degenerate equation $$\partial_t f(t,x)  - \partial_x \left( x^{\alpha} \partial_x f \right)(t,x) =0,$$ on the unit interval $x\in(0,1)$, in the strongly degenerate case $\alpha \in [1,2)$ with adapted boundary conditions at $x=0$ and boundary control at $x=1$. We use the flatness approach to construct explicit controls in some Gevrey classes steering the solution from any initial datum $f_0 \in L^2(0,1)$ to zero in any time $T>0$.
\end{abstract}

\textbf{Keywords--} partial differential equations; degenerate parabolic equation; boundary control; null-controllability; motion planning; flatness. 

\section{Introduction}
We consider the following control system
\begin{equation}
\left\{ \begin{array}{ll}
\partial_t f(t,x) -\partial_x \left( x^{\alpha} \partial_x \right) f(t,x)   = 0, &  (t,x) \in (0,T) \times (0,1), \\
\left( x^{\alpha} \partial_x\right) f(t,x) |_{{x=0}} = 0, & t\in (0,T), \\
f(t,1) = u(t), & t\in (0,T), \\
f(0,x) = f_0(x), & x\in (0,1),
\end{array} \right.
\label{eq:Cauchypbm}
\end{equation} where the state is the solution $f(t,x)$ and the control is the function $u(t)$. The parameter $\alpha \in [1,2)$ is fixed through the whole article. \par
The aim of this work is to construct explicit controls $u$ for the null-controllability of system (\ref{eq:Cauchypbm}) in finite time $T>0$, using the flatness method.

\subsection{Main result} 

We will make use of the Gevrey class of functions.

\begin{definition}
Let $s\in \R^+$ and $t_1,t_2\in \R$ with $t_1<t_2$. A function $h \in \C^{\infty}([t_1,t_2])$ is said to be Gevrey of order $s$ if 
\begin{equation*}
\exists M,R>0  \textrm{ such that } \sup_{t_1\leq r \leq t_2} \left| h^{(n)}(r) \right| \leq \frac{M (n!)^s}{R^n}.
\end{equation*} We then write $h \in \G^s([t_1,t_2])$.  
\end{definition} Before stating the main result, we have to recall the notion of weak solutions of the inhomogeneous system (\ref{eq:Cauchypbm}).

\begin{definition}[Weak solutions]
Let $f_0 \in L^2(0,1)$, $T>0$ and $u\in H^1(0,T)$. A weak solution of system (\ref{eq:Cauchypbm}) is a function $f\in \C^0([0,T];L^2(0,1))$ such that for every $t' \in [0,T]$ and for every 
\begin{equation}
\psi \in \C^1([0,t'];L^2(0,1)) \cap \C^0([0,t'];H^2(0,1))
\label{eq:regularitetest}
\end{equation} such that 
\begin{equation}
\label{eq:boundarytest}\left( x^{\alpha} \partial_x\right) \psi(t,x) |_{{x=0}} = \psi(t,1) = 0, \quad \forall t \in [0,t'],
\end{equation} one has 
\begin{gather}
\int_0^{t'} \int_0^1 f(t,x) \left(\partial_t \psi + \partial_x (x^{\alpha} \partial_x \psi ) \right)(t,x) \ud t \ud x \nonumber \\
=\int_0^1 f(t',x)\psi(t',x) \ud x - \int_0^1 f_0(x) \psi(0,x) \ud x + \int_0^{t'} u(t) \partial_x \psi (t,1) \ud t. \nonumber 
\end{gather}
\label{def:weak}
\end{definition}

As we show in Section 2 (see Corollary \ref{corollary:weak}), system (\ref{eq:Cauchypbm}) has a unique weak solution under suitable assumptions. Our main result is the following.

\begin{thm}
Let $f_0 \in L^2(0,1)$, $T>0$, $\tau \in (0,T)$ and $s \in (1,2)$. Then, there exists a flat output $y\in \G^s([\tau,T])$ such that the control
\begin{equation}
u(t) = \left\{ \begin{array}{ll}
0, & \textrm{if } t\in [0,\tau], \\
\sum_{k=0}^{\infty} \frac{y^{(k)}(t)}{(2-\alpha)^{2k}k! \prod_{j=1}^k\left( j + \frac{\alpha -1}{2 - \alpha} \right)}, &  \textrm{if } t\in (\tau,T],
\end{array} \right.
\label{eq:vraicontrole}
\end{equation} steers to zero at time $T$ the weak solution of system (\ref{eq:Cauchypbm}). Furthermore, the control $u$ belongs to $\G^s([0,T])$.
\label{thm:null-controllability}
\end{thm}

\subsection{Previous work}
\subsubsection{Null-controllability} 
The null-controllability of system 
\begin{equation*}
\left\{ \begin{array}{ll}
\partial_t f(t,x) -\partial_x \left( x^{\alpha} \partial_x \right) f(t,x)   = 1_{\omega}(x)v(t,x), &  (t,x) \in (0,T) \times (0,1), \\
\left( x^{\alpha} \partial_x\right) f(t,x) |_{{x=0}} = 0, & t\in (0,T), \\
f(t,1) = 0, & t\in (0,T), \\
f(0,x) = f_0(x), & x\in (0,1),
\end{array} \right.
\end{equation*} where $\omega \subset (0,1)$, has been studied by P. Cannarsa, P. Martinez and J. Vancostenoble in \cite{CMV}. Their strategy relies on appropriate Carleman estimates. To deal with the degeneracy at $\left\{x=0\right\}$, they use an adequate functional framework that we recall in Section 2, and Hardy-type inequalities. \par
 
The null-controllability of system (\ref{eq:Cauchypbm}) is a consequence of the internal null-controllability and the extension principle, since the control is located on $\left\{ x= 1 \right\}$, away from the degeneracy. The interest of the present article is to provide explicit controls.\par 
In the case of a control located on $\left\{x=0\right\}$, an approximate controllability result for $\alpha \in [0,1)$ has been proven by P. Cannarsa, J. Tort and M. Yamamoto in \cite{CTY} using Carleman estimates. The exact controllability was later proven by M. Gueye in \cite{Gueye} again in the weakly degenerate case $\alpha \in [0,1)$ by using the transmutation method.
\par
Other related one-dimensional problems have been treated: see \cite{CMV4, CMV5, ACF}, see \cite{CFR} for a non-divergence setting, see \cite{V} for a system with a singular potential. A multi-dimensional case has been studied in \cite{CMV2}. \par

\subsubsection{The flatness method} 
The main interest of the flatness method is to provide explicit controls. It has been developed for finite-dimensional systems (see \cite{Fliess}) and then generalised to some infinite-dimensional systems; see \cite{MRR} for the heat equation on a cylindrical domain with boundary control, \cite{MRR2} for one-dimensional parabolic equations with varying coefficients and \cite{MRR3} for the one-dimensional Schr\"{o}dinger equation. However, the strongly degenerate case $\alpha \in [1,2)$ considered in Theorem \ref{thm:null-controllability} does not belong to the class concerned in \cite{MRR2}. Our goal is to adapt the flatness method to this case.

\subsection{Open questions and perspectives}
The flatness method may also be successful on similar equations, for instance in non-divergence form as in \cite{CFR}. For the time being, this is an open problem. 

\subsection{Structure of the article}
In Section 2 we recall a well-posedness result and the functional framework. In Section 3 we derive, thanks to an heuristic method, an explicit solution of system (\ref{eq:Cauchypbm}) consisting on a formal series development. We prove its convergence, provided that the corresponding flat output is in a Gevrey class. In Section 4 we discuss the spectral analysis of the associated stationary problem. In Section 5 we study the regularising effect of system (\ref{eq:Cauchypbm}) when $u = 0$. In Section 6 we construct an appropriate flat output steering the solution of (\ref{eq:Cauchypbm}) to zero, which concludes the proof of Theorem \ref{thm:null-controllability}. Finally, we give in Appendices A and B a brief account of some results concerning the Gamma and Bessel functions needed in the proofs.

\subsection{Notation}
Since all the functions appearing in the article are real-valued, we omit any explicit mention by writing, for instance, $L^2(0,1)$ instead of $L^2((0,1);\R)$. If $h \in \C^k([t_1,t_2])$, for some $t_1,t_2 \in \R$ with $t_1<t_2$ and $k\in \N^*$, we will denote by $h'(t)$ and $h''(t)$ its first and second derivatives and by $h^{(n)}(t)$, for every $n\in \N$, $2 < n \leq k$, the $n-$th derivative. \par
If $h_1,h_2:\R \rightarrow \R$ are two real-valued functions and $\mu \in \overline{\R}$, we will write $h_1 \sim h_2$ as $x \rightarrow \mu$ to denote that $\lim_{t \rightarrow \mu} \frac{h_1(t)}{h_2(t)} =1$. \par
We will denote by $\langle\cdot, \cdot \rangle$ the inner product in $L^2(0,1)$.

\section{Well-posedness} 
We consider, for $T>0$ and $f_0 \in L^2(0,1)$, the following system 
\begin{equation}
\left\{ \begin{array}{ll}
\partial_t f(t,x) - \partial_x \left( x^{\alpha} \partial_x \right) f(t,x) = h(t,x) , & (t,x) \in (0,T) \times (0,1), \\
\left(x^{\alpha} \partial_x \right) f (t,x)|_{x=0} =0, & t \in (0,T), \\
f(t,1) = 0, & t\in (0,T), \\
f(0,x) = f_0(x), & x \in (0,1).
\end{array} \right.
\label{eq:homogeneoussystem}
\end{equation} We recall below a well-posedness result for system (\ref{eq:homogeneoussystem}) proven originally in \cite{CMV5}. The strategy of the proof consists in a semigroup approach and the introduction of adequate weighted Sobolev spaces, that we recall below. We refer to \cite{CMV5, CMP} for further details. \par
\bigskip
We introduce the weighted Sobolev space
\begin{eqnarray}
H_{\alpha}^1(0,1) := \left\{ f \in L^2(0,1); \, f \textrm{ is loc. absolutely continuous on } (0,1], \right. && \nonumber \\ 
 \left.  x^{\frac{\alpha}{2}} f' \in L^2(0,1) \textrm{ and } f(1)=0 \right\},&& \nonumber  
\end{eqnarray} endowed with the norm
\begin{equation*}
\|f\|_{H^1_{\alpha}(0,1)}^2 := \|f\|_{L^2(0,1)}^2 + \|x^{\frac{\alpha}{2}} f'\|_{L^2(0,1)}^2, \quad \forall f \in H^1_{\alpha}(0,1).
\end{equation*} We remark that $H^1_{\alpha}(0,1)$ is a Hilbert space with the scalar product
\begin{equation}
\langle f, g  \rangle_{H^1_{\alpha}} := \int_0^1 f(x) g(x) \ud x  + \int_0^1 x^{\alpha} f'(x) g'(x) \ud x, \quad  \forall f, g \in H^1_{\alpha}(0,1). 
\label{eq:scalarproduct}
\end{equation} 

\begin{proposition}[\cite{CMV5}, Proposition 3.2 and Theorem 3.1] 
Let
\begin{equation}
\left\{ \begin{array}{ll}
D(A):= \left\{ f \in H^1_{\alpha}(0,1); \, x^{\alpha}f'\in H^1(0,1) \right\}, \\
Af:= -(x^{\alpha} f')'.
\end{array} \right.
\label{eq:definitionofA}
\end{equation} Then, $A:D(A) \rightarrow L^2(0,1) $ is a closed self-adjoint positive operator with dense domain. As a consequence, $A$ is the infinitesimal generator of a strongly continuous semigroup, and for any $f_0 \in L^2(0,1)$, and $h\in L^2((0,T)\times(0,1))$ there exists a unique weak solution of system (\ref{eq:homogeneoussystem}), i.e., a function $f \in \C^0([0,T];L^2(0,1)) \cap L^2(0,T;H^1_{\alpha}(0,1))$  such that 
\begin{equation*}
f(t) = S(t)f_0 + \int_0^t S(t-s)h(s) \ud s, \quad \textrm{in } L^2(0,1), \quad \forall t \in [0,T].
\end{equation*}
\label{proposition:Aisselfadjoint}
\end{proposition} 
As a consequence, using classical arguments (see for instance \cite[Section 2.5.3]{Coron}), we deduce the following result. 
\begin{corollary}
Let $T>0$, $f_0 \in L^2(0,1)$ and $u \in H^1(0,T)$. Then, system (\ref{eq:Cauchypbm}) has a unique weak solution (see Definition \ref{def:weak}).
\label{corollary:weak}
\end{corollary}

\begin{proof}
Let $f_0 \in L^2(0,1)$, $u \in H^1(0,T)$ and 
\begin{equation*}
\theta(x) := x^2, \quad x\in [0,1]. 
\end{equation*}We consider the system 
\begin{equation*}
\left\{ \begin{array}{ll}
\left( \partial_t  - \partial_x \left( x^{\alpha} \partial_x  \right) \right)g (t,x) = H(t,x), & (t,x) \in (0,T) \times (0,1), \\
\left(x^{\alpha} \partial_x \right) g (t,x)|_{x=0} =0, & t \in (0,T), \\
g(t,1) = 0, & t\in (0,T), \\
g(0,x) = f_0(x) - u(0)\theta(x), & x \in (0,1),
\end{array} \right.
\end{equation*} with 
\begin{equation*}
H(t,x) : = -u'(t)\theta(x) - u(t) A\theta(x), \quad \forall (t,x) \in (0,T)\times (0,1).
\end{equation*} Since $H \in L^2((0,T)\times(0,1))$, by Proposition \ref{proposition:Aisselfadjoint} there exists a unique weak solution $g \in \C^0([0,T];L^2(0,1)) \cap L^2(0,T;H^1_{\alpha}(0,1))$ of this problem. We set
\begin{equation*}
f(t,x) := g(t,x) + u(t)\theta(x).
\end{equation*} Then, using the integral formulation associated to $g$, one shows that $f$ is a weak solution of system (\ref{eq:Cauchypbm}) in the sense of Definition \ref{def:weak}. \par 
The uniqueness follows since, if $f_1$ and $f_2$ are weak solutions of (\ref{eq:Cauchypbm}), then $f_1 - f_2$ is the unique weak solution of system (\ref{eq:homogeneoussystem}) with $h\equiv 0$, and then by Proposition \ref{proposition:Aisselfadjoint}, $f_1 - f_2 = 0$. 
\end{proof}

\section{Explicit solution}
\subsection{Heuristics}
We consider the following formal expansion
\begin{equation*}
f(t,x) = \sum_{k=0}^{\infty} c_{2k}(t) \left( x^{1-\frac{\alpha}{2}}\right)^{2k}, \quad \forall (t,x) \in (0,T)\times (0,1).
\end{equation*} where $(c_{2k}(t))_{k\in \N}$ is a sequence of real numbers. We formally have
\begin{eqnarray}
\partial_x \left( x^{\alpha} \partial_x f \right) (t,x) &=& \sum_{k=0}^{\infty} c_{2(k+1)}(t) (2-\alpha)^2(k+1)\left[ k +1 +\frac{\alpha -1}{2 - \alpha} \right]\left( x^{1-\frac{\alpha}{2}}\right)^{2k}, \nonumber \\
\partial_t f(t,x) &=& \sum_{k=0}^{\infty} c_{2k}'(t) \left( x^{1 - \frac{\alpha}{2}} \right)^{2k}. \nonumber 
\end{eqnarray} If $f$ solves (\ref{eq:Cauchypbm}), then the following recurrence relation holds 
\begin{equation*}
c_{2(k+1)}(t) = \frac{c_{2k}'(t)}{(2-\alpha)^2 (k+1) \left( k + 1 + \frac{\alpha -1}{2 - \alpha} \right)} , \quad \forall k \in \N.
\end{equation*} Choosing a flat output $ c_0(t) := y(t)$ and iterating, we readily have
\begin{equation*}
c_{2k}(t) = \frac{ y^{(k)}(t)}{(2-\alpha)^{2k} k! \prod_{j=1}^k \left( j + \frac{\alpha -1}{2 -\alpha}  \right)}, \quad \forall t \in (0,T), \, \forall k \in \N.
\end{equation*} This gives a formal solution of (\ref{eq:Cauchypbm}),
\begin{equation}
f(t,x) = \sum_{k=0}^{\infty} \frac{ y^{(k)}(t) \left( x^{1 - \frac{\alpha}{2}} \right)^{2k}}{(2-\alpha)^{2k} k! \prod_{j=1}^k \left( j + \frac{\alpha-1}{2 -\alpha}  \right)},
\label{eq:explicitsolution}
\end{equation} and a control given by $u(t) = f(t,1)$, which is
\begin{equation}
u(t) = \sum_{k=0}^{\infty} \frac{ y^{(k)}(t)}{(2-\alpha)^{2k} k! \prod_{j=1}^k \left( j + \frac{\alpha -1}{2 -\alpha}  \right)}.
\label{eq:explicitcontrol}
\end{equation}

\subsection{Pointwise solutions}
The goal of this section is to introduce a notion of pointwise solution of system (\ref{eq:Cauchypbm}) to give a sense to the heuristics made in the previous section. \par
We define
\begin{equation*}
\C^2_{\alpha}(0,1) := \left\{ f \in \C^0([0,1]) \cap \C^2((0,1)) \textrm{ such that } x^{\alpha}f'(x) \in \C^0([0,1))  \right\}.
\end{equation*}

\begin{definition}[Pointwise solution] Let $t_1,t_2 \in \R$ with $t_1 < t_2$. Let $f_{t_1}\in \C^0(0,1)$ and $u \in \C^0([t_1,t_2])$. A pointwise solution of system 
\begin{equation}
\left\{ \begin{array}{ll}
\partial_t f(t,x) - \partial_x \left( x^{\alpha} \partial_x f \right)(t,x) = 0, & (t,x) \in (t_1,t_2) \times (0,1), \\
x^{\alpha} \partial_x f(t,x)|_{x=0} = 0, & t \in (t_1,t_2), \\
f(t,1) =u(t), & t\in (t_1,t_2), \\
f(t_1,x) = f_{t_1}(x), & x \in (0,1),
\end{array} \right.
\label{eq:pointwiseCauchypbm}
\end{equation}
is a function $f \in \C^0([t_1,t_2]\times[0,1]) \cap \C^1((t_1,t_2)\times (0,1))$ such that 
\begin{enumerate}
\item $f(t,\cdot) \in \C^2_{\alpha}(0,1)$, $\forall t \in (t_1,t_2)$,  
\item $\partial_t f -\partial_x (x^{\alpha} \partial_x f) = 0 $  pointwisely in $(t_1,t_2)\times (0,1)$,
\item $\lim_{x\rightarrow 0^+} x^{\alpha} \partial_x f(t,x) = 0$, $\forall t \in (t_1,t_2)$, 
\item $f(t,1) = u(t)$, $\forall t \in (t_1,t_2)$,
\item $f(t_1,x) = f_{t_1}(x)$, $\forall x\in (0,1)$.
\end{enumerate}
\label{def:pointwise}
\end{definition}

\begin{remark}
The usual energy argument proves that, given $u\in \C^0([t_1,t_2])$, the pointwise solution of system (\ref{eq:pointwiseCauchypbm}) is unique. We observe that, changing parameters adequately in Definition \ref{def:weak} a pointwise solution of  (\ref{eq:pointwiseCauchypbm}) is also a weak solution.
\label{remark:uniqueness}
\end{remark}

\subsection{Convergence}
The goal of this section is the proof of the following result.

\begin{proposition}
Let $t_1, t_2 \in \R$, with $t_1<t_2$. If $y \in \G^s([t_1,t_2])$ for some $s \in (0,2)$, then 
\begin{enumerate}
\item the control $u$ given by (\ref{eq:explicitcontrol}) is well defined and belongs to $\G^s([t_1,t_2])$,
\item the function given by (\ref{eq:explicitsolution}) is a pointwise solution (see Definition \ref{def:pointwise}) of system (\ref{eq:pointwiseCauchypbm}) in $(t_1,t_2)\times (0,1)$ with $u$ given by (\ref{eq:explicitcontrol}) and initial datum 
\begin{equation*}
f_{t_1}(x) := \sum_{k=1}^{\infty} \frac{y^{(k)}(t_1) \left( x^{1 - \frac{\alpha}{2}} \right)^{2k}}{(2-\alpha)^{2k} k! \prod_{j=1}^k \left( j + \frac{\alpha-1}{2-\alpha} \right)}, \quad \forall x\in [0,1].
\end{equation*}
\end{enumerate}
\label{proposition:convergence}
\end{proposition}

\begin{proof}
Let $M,R>0$ be such that $|y^{(n)}(t)| \leq \frac{M n!^s}{R^n},$ for any $n \in \N, \, t \in [t_1,t_2]$.

\begin{description}
\item[Step 1] We prove that $u$ is well defined and belongs to $\C^{\infty}([t_1,t_2])$. \par
For any $t \in [t_1,t_2]$, $k\in \N^*$, we have, as $\frac{\alpha-1}{2-\alpha}\geq 0$,
\begin{equation*}
\frac{|y^{(k)}(t)|}{(2-\alpha)^{2k}k!\prod_{j=1}^k \left( j + \frac{\alpha-1}{2-\alpha}  \right)}  \leq  \frac{M k!^s}{R^k (2-\alpha)^{2k} k!^2} = \frac{M}{R^k (2-\alpha)^{2k} k!^{2-s}}.
\end{equation*} Hence, the series in (\ref{eq:explicitcontrol}) converges uniformly w.r.t. $t \in [t_1,t_2]$ and $u \in \C^0([t_1,t_2])$. Furthermore, for any $n\in \N^*$, the function $\xi_{n,k}(t):= \frac{y^{(k+n)}(t)}{(2-\alpha)^{2k}k!\prod_{j=1}^k \left( j + \frac{\alpha-1}{2-\alpha}  \right)}$ satisfies  
\begin{equation*}
|\xi_{n,k}(t)| \leq  \frac{M (k+n)!^s}{R^{n+k}(2-\alpha)^{2k} k!^2}, \quad \forall t \in [t_1,t_2], \, k,n \in \N.
\end{equation*} Thus, $\sum_k \xi_{n,k}(t) $ converges uniformly w.r.t $t \in [t_1,t_2]$. Whence, $u \in \C^{\infty}([t_1,t_2])$ and for every $ n \in \N $, $t \in [t_1,t_2]$, $u^{(n)}(t) = \sum_{k=0}^{\infty} \xi_{n,k}(t).$

\item[Step 2] We prove that $u$ is Gevrey of order $s$. \par 

Let $n\in \N$. We deduce from last inequality that 
\begin{eqnarray}
\left|  u^{(n)}(t) \right| & \leq & \sum_{k=0}^{\infty} \frac{M (k+n)!^s}{R^{n+k}(2-\alpha)^{2k} k!^2  }  \nonumber \\
& \leq & M\left[ \sum_{k=0}^{\infty} \frac{1}{(k!)^{2-s} } \left( \frac{2^s}{R(2-\alpha)^2} \right)^k \right] \left( \frac{2^s}{R} \right)^n n!^s , \label{eq:estimationGevrey} 
\end{eqnarray} where we have used (\ref{eq:factorials}). The D'Alembert criterium for entire series shows that, whenever $s \in (0,2)$, the series above converges, which shows that $u \in \G^s([t_1,t_2]).$

\item[Step 3] We show that the function $f$ given by (\ref{eq:explicitsolution}) is well defined and $f \in \C^0([t_1,t_2]\times[0,1]) \cap \C^1((t_1,t_2)\times(0,1))$. \par
Let, for every $k\in \N$,
\begin{equation*}
f_k(t,x) := \frac{y^{(k)}(t) \left(x^{1-\frac{\alpha}{2}}\right)^{2k} }{(2-\alpha)^{2k}k!\prod_{j=1}^k \left( j + \frac{\alpha-1}{2-\alpha}  \right)}, \quad \forall (t,x) \in [t_1,t_2]\times [0,1].
\end{equation*}
Then,
\begin{equation*}
|f_k(t,x)| \leq  \frac{M}{k!^{2-s}} \left( \frac{1}{R(2-\alpha)} \right)^k, \quad \forall (t,x) \in [t_1,t_2]\times[0,1].
\end{equation*}  This proves that $\sum_k f_k$ converges uniformly w.r.t. $(t,x) \in [t_1,t_2]\times [0,1]$. Thus, $f \in \C^0([t_1,t_2]\times[0,1])$. \par
We observe that $\exists k_0=k_0(\alpha)\in \N^*$  such that $\left(2 - \alpha \right) k_0 \geq 1.$ Then, for every $k\geq k_0$, $f_k(t,\cdot) \in \C^1([0,1])$ and
\begin{eqnarray}
\left| \partial_x f_k(t,x) \right| &=& \left| \frac{y^{(k)}(t)2k\left( 1 - \frac{\alpha}{2} \right) x^{-\frac{\alpha}{2}} \left( x^{1 - \frac{\alpha}{2}} \right)^{2k-1} }{(2-\alpha)^{2k} k! \prod_{j=1}^k \left( j + \frac{\alpha -1 }{2 -\alpha} \right)  }  \right| \nonumber \\
& \leq & 2M \left(1 - \frac{\alpha}{2}  \right) \frac{k}{k!^{2-s}} \left( \frac{1}{R(2-\alpha)^2}  \right)^k, \quad \forall x \in [0,1], \nonumber 
\end{eqnarray} since $\left( 1 -\frac{\alpha}{2} \right) (2k-1) - \frac{\alpha}{2} \geq 0$. This proves that $\sum_{k \geq k_0} \partial_x f_k$ converges uniformly w.r.t. $(t,x) \in [t_1,t_2] \times [0,1]$. Thus, $f(t,\cdot) \in \C^1((0,1])$ for every $t\in [t_1,t_2]$. Note that $f$ may not be differentiable w.r.t. $x$ at $x=0$ because of the finite number of terms $\sum_{k=0}^{k_0} \partial_x f_k$. Moreover, $\partial_x f (t,x) = \sum_{k=0}^{\infty} \partial_x f_k(t,x)$ for every $(t,x) \in (t_1,t_2) \times (0,1).$ \par 
A similar argument shows that, for every $x \in (0,1),$ $f(\cdot,x) \in \C^1(t_1,t_2)$ and 
\begin{equation}
\partial_t f(t,x) = \sum_{k=0}^{\infty}\partial_t f_k(t,x), \quad \forall (t,x) \in (t_1,t_2) \times(0,1). 
\label{eq:deriveentemps}
\end{equation}
Finally, since the partial derivatives of $f$ exist and are continuous in $(t_1,t_2)\times(0,1)$, $f \in \C^1((t_1,t_2)\times(0,1))$. \par

\item[Step 4] We show that $f(t,\cdot) \in \C^2_{\alpha}(0,1)$, for every $t\in (t_1,t_2)$. \par 
Let $k_1= k_1(\alpha) \in \N^*$ such that $k_1(2-\alpha) \geq 2.$ Working as in Step 3, we see that $\sum_{k\geq k_1} \partial^2_x f_k$ converges uniformly w.r.t. $(t,x) \in (t_1,t_2) \times (0,1)$. Thus, $f(t,\cdot) \in \C^2(0,1),$ $\forall t \in (t_1,t_2)$. Furthermore, 
\begin{equation}
\partial_x \left( x^{\alpha} \partial_x f \right)(t,x) = \sum_{k=1}^{\infty} \frac{y^{(k)}(t) \left(x^{1-\frac{\alpha}{2}} \right)^{2(k-1)}}{(2-\alpha)^{2(k-1)}(k-1)!\prod_{j=1}^{k-1}\left( j + \frac{\alpha -1}{2-\alpha} \right)}.
\label{eq:deriveesegonde}
\end{equation} for every $(t,x) \in (t_1,t_2) \times(0,1)$. From Step 3, we obtain 
\begin{eqnarray}
\left| x^{\alpha}\partial_x f(t,x) \right| &=& \left| \sum_{k=1}^{\infty} \frac{y^{(k)}(t) 2k\left( 1 - \frac{\alpha}{2} \right) x^{2k\left(1 - \frac{\alpha}{2} \right) +\alpha -1} }{(2-\alpha)^{2k} k! \prod_{j=1}^k \left( j + \frac{\alpha -1 }{2 -\alpha} \right)  }  \right| \nonumber \\
& \leq & 2M \left(1 - \frac{\alpha}{2}  \right) \sum_{k=1}^{\infty} \left[ \frac{k}{k!^{2-s}} \left( \frac{1}{R(2-\alpha)^2}  \right)^k \right]x, \nonumber 
\end{eqnarray} for all $(t,x) \in (t_1,t_2) \times (0,1)$, which implies, since $\alpha \in [1,2)$, that
\begin{equation*}
x^{\alpha} \partial_x f(t,x) \rightarrow 0, \quad \textrm{ as } x \rightarrow 0^+.
\end{equation*} Therefore, $f(t,\cdot) \in \C^2_{\alpha}$, for every $t \in (t_1,t_2)$.

\item[Step 5] According to (\ref{eq:deriveentemps}) and (\ref{eq:deriveesegonde}), an straightforward computation shows that the equation in (\ref{eq:pointwiseCauchypbm}) is satisfied.

\end{description}

\end{proof}

\section{Spectral Analysis}

The goal of this section is to give the explicit expression of the eigenfunctions and eigenvalues of the spectral problem 
\begin{equation}
\left\{ \begin{array}{ll}
A \varphi (x) = \lambda \varphi(x), & x \in (0,1), \\
\left( x^{\alpha} \varphi' \right)|_{x=0} = \varphi(1) = 0, &
\end{array} \right.
\label{eq:spectralproblem}
\end{equation} where $A$ is given by (\ref{eq:definitionofA}). We will make use of several results about Bessel functions recalled in Appendix B. Form now on, we use the notation  
\begin{equation}
\nu := \frac{\alpha-1}{2-\alpha}.
\label{eq:nu}
\end{equation}

\begin{proposition}
Let  
\begin{equation}
\varphi_k(x) = \frac{\sqrt{2-\alpha}}{|J_{\nu + 1}(j_{\nu,k})|} x^{\frac{1-\alpha}{2}} J_{\nu} \left( j_{\nu,k} x^{1 - \frac{\alpha}{2}} \right), \quad \forall x\in (0,1), \, k \in \N^*.
\label{eq:eigenvectors}
\end{equation} Then,
\begin{enumerate}
\item $\varphi_k \in D(A),$ $\forall k \in \N^*$,
\item $\varphi_k$ satisfies (\ref{eq:spectralproblem}) with 
\begin{equation}
\lambda_k := \left( 1 - \frac{\alpha}{2} \right)^2 j_{\nu,k}^2, \quad \forall k\in \N^*,
\label{eq:eigenvalues}
\end{equation} 
\item $\left( \varphi_k \right)_{k \in \N^*}$ is a Hilbert basis of $L^2(0,1)$,
\item for every $f_0 \in L^2(0,1)$ the solution of (\ref{eq:homogeneoussystem}) with $h=0$ writes
\begin{equation}
f(t) = \sum_{k=1}^{\infty} e^{-\lambda_k t} \langle f_0, \varphi_k \rangle \varphi_k \quad \textrm{in } L^2(0,1), \, \forall t\in [0,T].
\label{eq:Fourierexpansion}
\end{equation}
\end{enumerate}
\label{proposition:spectralfamily}
\end{proposition}

\begin{proof} We will note for simplicity $b_k:= \frac{\sqrt{2-\alpha}}{|J_{\nu+1}(j_{\nu,k})|}$ and $\tilde{\varphi}_k := \frac{1}{b_k} \varphi_k$, for every $ k \in \N^*$. 

\begin{description}

\item[Step 1] We prove that $\varphi_k \in D(A)$, for every $k\in \N^*$ and that $A\varphi_k - \lambda_k\varphi_k =0 $. \par
Let $k\in \N^*$. We observe that $\varphi_k \in \C^{\infty}((0,1]) \cap \C^0([0,1])$, for any $k\in \N^*$ and $x \in (0,1)$. We have
\begin{equation}
\tilde{\varphi}_k'(x) = \frac{1-\alpha}{2} x^{-\frac{1+\alpha}{2}} J_{\nu} (j_{\nu,k}x^{1-\frac{\alpha}{2}}) + j_{\nu,k}\left( 1 - \frac{\alpha}{2}\right)x^{\frac{1}{2}-\alpha} J_{\nu}'(j_{\nu,k}x^{1-\frac{\alpha}{2}}).
\label{eq:firstderivative}
\end{equation} Whence, using (\ref{eq:Besselrecurrencerelation}) and Lemma \ref{lemma:asymptoticBesselatzero}, we deduce
\begin{equation*}
x^{\frac{\alpha}{2}} \tilde{\varphi}_k' =  (1-\alpha) \underset{x\rightarrow 0^+}O\left( x^{\frac{\alpha}{2}-1} \right) + \underset{x\rightarrow 0^+}O\left( x^{1-\frac{\alpha}{2}} \right).
\end{equation*} It follows that $x^{\frac{\alpha}{2}}\varphi_n' \in L^2(0,1)$. Thus $\varphi_k \in H^1_{\alpha}(0,1)$. Moreover, from (\ref{eq:firstderivative}), a direct computation shows
\begin{eqnarray}
\left( x^{\alpha} \tilde{\varphi}_k'\right)' &= & -\left( \frac{1 -\alpha}{2} \right)^2 x^{\frac{\alpha - 3}{2}} J_{\nu}(j_{\nu,k}x^{1 - \frac{\alpha}{2}}) \nonumber \\
&& + \left(1 - \frac{\alpha}{2} \right)^2 j_{\nu,k } x^{-\frac{1}{2}} J_{\nu}'(j_{\nu,k}x^{1-\frac{\alpha}{2}}) \nonumber \\
&& + \left( 1 - \frac{\alpha}{2}\right)^2 j_{\nu,k}^2 x^{\frac{1-\alpha}{2}}  J_{\nu}''(j_{\nu,k}x^{1-\frac{\alpha}{2}}). \label{eq:secondderivative} 
\end{eqnarray}
Then, evaluating equation (\ref{eq:Besseldifferentialequation}) at $z= j_{\nu,k}x^{1-\frac{\alpha}{2}}$ and multiplying by $x^{\frac{\alpha-3}{2}}$, it follows  
\begin{eqnarray}
&& j_{\nu,k}^2 x^{\frac{1-\alpha}{2}}J_{\nu}''(j_{\nu,k}x^{1-\frac{\alpha}{2}})  \nonumber \\
&& \quad = - j_{\nu,k} x^{-\frac{1}{2}} J_{\nu}'(j_{\nu,k}x^{1-\frac{\alpha}{2}}) - j_{\nu,k}^2 x^{\frac{1-\alpha}{2}}J_{\nu}(j_{\nu,k}x^{1-\frac{\alpha}{2}}) \nonumber \\
&& \quad \quad + \left( \frac{\alpha -1}{2-\alpha} \right)^2 x^{\frac{\alpha -3}{2}} J_{\nu}(j_{\nu,k} x^{1-\frac{\alpha}{2}}). \nonumber
\end{eqnarray} Substituting in (\ref{eq:secondderivative}), this gives
\begin{equation*}
-\left( x^{\alpha} \tilde{\varphi}_k'\right)' =  \left(1 - \frac{\alpha}{2} \right)^2 j_{\nu,k}^2 x^{\frac{1-\alpha}{2}} J_{\nu}(j_{\nu,k}x^{1-\frac{\alpha}{2}}) = \lambda_k \tilde{\varphi}_k. 
\end{equation*} Then, we readily have $\left( x^{\alpha} \tilde{\varphi}_k'\right)' \in H^1_{\alpha}(0,1) \subset L^2(0,1)$. Thus, $\varphi_k \in D(A)$. Moreover, $A\varphi_k = \lambda_k \varphi_k$.

\item[Step 2] We check the boundary condition of (\ref{eq:spectralproblem}) at $x =0$. \par
 We observe first that the case $\alpha=1$ is straightforward. From (\ref{eq:firstderivative}), (\ref{eq:Besselrecurrencerelation}) and Lemma \ref{lemma:asymptoticBesselatzero}, we have 
\begin{equation*}
|x^{\alpha} \tilde{\varphi}_n'(x)| = \underset{x\rightarrow 0^+}O\left( x^{\alpha -1} \right). 
\end{equation*} Then, it follows that $\lim_{x\rightarrow 0^{+}} x^{\alpha} \tilde{\varphi}_n'(x) = 0$. This shows, combined with Step 1, that $\varphi_k$ satisfies (\ref{eq:spectralproblem}).

\item[Step 3] We prove that $\left( \varphi_k \right)_{k\in \N^*}$ is an orthonormal family in $L^2(0,1)$. \par
Let $n,m \in \N^*$. Then, changing variables and using (\ref{eq:Besselorthogonality}), we get

\begin{eqnarray}
&& \int_0^1 \varphi_n(x) \varphi_m(x) \ud x  \nonumber \\
&& = (2-\alpha) \int_0^1 x^{1-\alpha} \frac{J_{\nu}(j_{\nu,n} x^{1- \frac{\alpha}{2}}) }{|J_{\nu+1}(j_{\nu,n})|} \frac{J_{\nu}(j_{\nu,m} x^{1- \frac{\alpha}{2}})}{|J_{\nu+1}(j_{\nu,m})|} \ud x \nonumber \\
&& = \frac{2}{|J_{\nu+1}(j_{\nu,n})||J_{\nu+1}(j_{\nu,m})|} \int_0^1 y  J_{\nu}(j_{\nu,n} y)J_{\nu}(j_{\nu,m} y) \ud y = \delta_{n,m}, \nonumber 
\end{eqnarray} where $\delta_{n,m}$ stands for the Kronecker delta.

\item[Step 4] We prove that $\left(\varphi_k\right)_{k \in \N^*}$ is a Hilbert basis of $L^2(0,1)$ by checking the Bessel equality. Let $f \in L^2(0,1)$ and let 
\begin{equation}
a_k := \int_0^1 f(x) \varphi_k(x) \ud x, \quad \forall k \in \N^*.
\label{eq:Fouriercoefficients}
\end{equation} Then, using Lemma \ref{lemma:Besselbase} and changing variables twice, we get
\begin{eqnarray}
\sum_{k=1}^{\infty} |a_k|^2 &=& \sum_{k=1}^{\infty} \left| \int_0^1 f(x) \frac{\sqrt{2-\alpha}}{|J_{\nu+1}(j_{\nu,k})|}x^{\frac{1-\alpha}{2}} J_{\nu}\left(j_{\nu,k}x^{1 - \frac{\alpha}{2}}\right) \ud x \right|^2 \nonumber \\
& = & \frac{2}{2-\alpha} \sum_{k=1}^{\infty} \left| \int_0^1 y^{\frac{\alpha-1}{2-\alpha} + \frac{1}{2}} f(y^{\frac{2}{2-\alpha}}) \frac{\sqrt{2y}}{|J_{\nu+1}(j_{\nu,k})|} J_{\nu}(j_{\nu,k}y) \ud y \right| ^2 \nonumber \\
& = & \frac{2}{2-\alpha} \int_0^1 y^{\frac{2(\alpha -1)}{2-\alpha} +1 } \left| f(y^{\frac{2}{2-\alpha}})\right|^2 \ud y \nonumber \\
& = & \int_0^1 |f(z)|^2 \ud z = \|f\|_{L^2(0,1)}^2. \nonumber 
\end{eqnarray} 

\item[Step 5] Finally, (\ref{eq:Fourierexpansion}) is a consequence of \cite[Theorem 8.2.3, pp.237--240]{Allaire}.
\end{description}
\end{proof}

\section{Regularising effect}
We use the orthonormal basis obtained in Proposition \ref{proposition:spectralfamily} and some properties of Bessel functions to quantify the smoothing of the solution of system (\ref{eq:Cauchypbm}) when $u\equiv 0$. \par

\begin{proposition} Let $f_0 \in L^2(0,1)$, $T>0$ and let $f\in \C^0([0,T];L^2(0,1))$ be the unique weak solution of system (\ref{eq:homogeneoussystem}) when $h=0$, according to Proposition \ref{proposition:Aisselfadjoint}. Then, there exists $Y \in \C^{\infty}((0,T])$ such that for every $\sigma \in (0,T)$,
\begin{equation*}
Y \in \G^1([\sigma,T])
\end{equation*} and 
\begin{equation}
f(t,x) = \sum_{n=0}^{\infty} \frac{Y^{(n)}(t)  \left( x^{1- \frac{\alpha}{2}} \right)^{2n}}{(2-\alpha)^{2n} n! \prod_{j=1}^n \left( j + \frac{\alpha - 1}{2 - \alpha} \right)}, \quad \forall (t,x) \in [\sigma,T]\times[0,1]. 
\label{eq:trueexplicitsolution}
\end{equation} Moreover, $f$ solves system (\ref{eq:pointwiseCauchypbm}) pointwisely (see Definition \ref{def:pointwise}) in $(\sigma,T)\times(0,1)$ with $u =0$ and initial datum $f_{\sigma}(x) = f(\sigma,x)$.  
\label{proposition:regularisingeffect}
\end{proposition}

\begin{proof}
Let $\nu$ be given by (\ref{eq:nu}) and $a_k$ as in (\ref{eq:Fouriercoefficients}). Let $\sigma \in (0,T)$ be fixed but arbitrary. Let $t\in[\sigma,T$] be fixed. By (\ref{eq:Fourierexpansion}) and (\ref{eq:Besselseries}), we have, for a.e. $x\in [0,1]$,
\begin{eqnarray}
f(t,x) &=& \sum_{k=1}^{\infty} e^{-\lambda_k t} \frac{a_k \sqrt{2-\alpha}}{|J_{\nu +1}(j_{\nu,k})|} x^{\frac{1-\alpha}{2}} J_{\nu}\left( j_{\nu,n} x^{1-\frac{\alpha}{2}} \right) \nonumber \\
& = & \sum_{k=1}^{\infty} e^{-\lambda_k t} \frac{a_k \sqrt{2-\alpha} }{|J_{\nu +1}(j_{\nu,k})|}x^{\frac{1-\alpha}{2}} \sum_{n=0}^{\infty} \frac{(-1)^n}{n! \Gamma\left(n + 1 + \nu \right)}\left( \frac{j_{\nu,k} x^{1 - \frac{\alpha}{2}}}{2} \right)^{2n + \nu} \nonumber \\
& = & \sum_{k=1}^{\infty} \sum_{n=0}^{\infty} B_{n,k}(t,x), \label{eq:Fubini} 
\end{eqnarray} where, for every $(n,k)\in \N \times \N^*$,
\begin{equation*}
B_{n,k}(t,x) := e^{-\lambda_k t} b_k \frac{(-1)^n j_{\nu,k}^{2n + \nu}}{n! \Gamma (n +1 + \nu) 2^{2n + \nu}} \frac{\left( x^{1 - \frac{\alpha}{2}} \right)^{2n}}{|J_{\nu+1}(j_{\nu,k})|},
\end{equation*} and $b_k:=a_k\sqrt{2-\alpha}$, $\forall k\in \N^*$. 
\begin{description}

\item[Step 1] We show that 
\begin{equation}
\sum_{n=0}^{\infty} \left( \sum_{k=1}^{\infty} |B_{n,k}(t,x)| \right) < \infty, \quad \forall x \in [0,1].
\label{eq:convergencevaleurabsolue}
\end{equation} Indeed, since $\lambda_k>0$, we have for every $(n,k) \in \N \times \N^*$ and $x \in [0,1]$,
\begin{eqnarray}
\left|B_{n,k}(t,x)\right| & \leq & \frac{|b_k|j_{\nu,k}^{2n + \nu} e^{-\lambda_k \sigma}}{2^{2n + \nu}n! \Gamma (n +1 + \nu)|J_{\nu +1}(j_{\nu,k})|}  \nonumber \\
& \leq &  \frac{C_1 |b_k| e^{-\lambda_k \sigma} j_{\nu,k}^{2n + \nu + \frac{1}{2}}}{2^{2n}n! \Gamma (n +1 + \nu)} , \label{eq:majorationB}
\end{eqnarray} for a constant $C_1 >0$, using Lemma \ref{lemma:asymptoticBesselaux}. \par
We fix $n\in \N$ and we define the function $h_n^{\alpha} \in \C^{\infty}(\R^+;\R^+)$ by
\begin{equation*}
h_n^{\alpha}(x) := e^{-\left( 1 - \frac{\alpha}{2}\right)^2 x^2 \sigma} x^{2n + \nu + \frac{1}{2}}, \quad \forall x \in [0, +\infty), 
\end{equation*} which satisfies that
\begin{equation}
\frac{\ud }{\ud x}h_n^{\alpha}(x) >0, \, \forall x\in (0,N_n^{\alpha}) \quad \textrm{and} \quad \frac{\ud }{\ud x} h_n^{\alpha}(x) <0, \, \forall x \in (N_n^{\alpha}, \infty), 
\label{eq:monotonicityofh}
\end{equation} where $N_n^{\alpha} := \frac{2}{2-\alpha}\sqrt{\frac{1}{\sigma}\left( n + \frac{\alpha}{4(2-\alpha)}\right)}$. Hence, from (\ref{eq:majorationB}) and (\ref{eq:eigenvalues}), 
\begin{equation}
\sum_{k=1}^{\infty} | B_{n,k}(t,x) |\leq  \frac{C_1 \sup_{k}|b_k|}{2^{2n} n! \Gamma (n + 1 +\nu) } \sum_{k=1}^{\infty} h_n^{\alpha}(j_{\nu,k}) 
\label{eq:seriesBnk}
\end{equation} Introducing $K_n^{\alpha} := \sup \left\{ k \in \N^*; j_{\nu,k} \leq N_n^{\alpha} \right\}$, we write 
\begin{equation}
\sum_{k=1}^{\infty} h_n^{\alpha} (j_{\nu,k}) = h_n^{\alpha}(j_{\nu,K_n^{\alpha}}) + h_n^{\alpha} (j_{\nu,K_n^{\alpha}+1}) +\sum_{k \in \N^* - \left\{ K_n^{\alpha}, K_n^{\alpha}+1 \right\}} h_n^{\alpha}(j_{\nu,k})
\label{eq:decoupage}
\end{equation} On one hand, we have
\begin{eqnarray}
&& h_n^{\alpha}(j_{\nu,K_n^{\alpha}}) + h_n^{\alpha}(j_{\nu,K_n^{\alpha}+1}) \leq 2 h_n^{\alpha}(N_n^{\alpha})\nonumber \\
&& \quad \quad \leq 2e^{-\left( n + \frac{\alpha}{4(2-\alpha)}\right)} \left( n + \frac{\alpha}{4(2-\alpha)}\right)^{n + \frac{\alpha}{4(2-\alpha)}}\left[\frac{1}{\sigma}\left(\frac{2}{2-\alpha} \right)^2 \right]^{n + \frac{\alpha}{4(2-\alpha)}} \nonumber \\
&& \quad \quad \leq C_2 \Gamma \left( n + \frac{\alpha}{4(2-\alpha)} + \frac{1}{2} \right) \left[\frac{1}{\sigma}\left(\frac{2}{2-\alpha} \right)^2 \right]^{n + \frac{\alpha}{4(2-\alpha)}}, \label{eq:majorationpremierterme} 
\end{eqnarray} for a constant $C_2>0$, using Lemma \ref{lemma:Stirling} with $a =1, b=\frac{1}{2}$. On the other hand, using (\ref{eq:monotonicityofh}), we write
\begin{eqnarray}
&& \sum_{k \in \N^*-\left\{ K_n^{\alpha}, K_n^{\alpha} + 1 \right\}} h_n^{\alpha}(j_{\nu,k}) \leq \nonumber \\
&& \leq \sum_{k=1}^{K_n^{\alpha}-1} \frac{1}{j_{\nu,k+1} - j_{\nu,k}} \int_{j_{\nu,k}}^{j_{\nu,k+1}} h_n^{\alpha}(x) \ud x + \sum_{K_n^{\alpha}+1}^{\infty} \frac{1}{j_{\nu,k} - j_{\nu,k-1}} \int_{j_{\nu,k-1}}^{j_{\nu,k}}h_n^{\alpha}(x) \ud x \nonumber \\
&& \leq \sup_{k\in \N^*} \left\{ \frac{1}{j_{\nu,k+1} - j_{\nu,k}} \right\} \left( \int_{j_{\nu,1}}^{j_{\nu,K_n^{\alpha}}} h_n^{\alpha}(x) \ud x + \int_{j_{\nu,K_n+1}^{\alpha}}^{\infty} h_n^{\alpha}(x) \ud x \right) \nonumber \\
&& \leq C_3 \int_0^{\infty} h_n^{\alpha} (x) \ud x, \nonumber  
\end{eqnarray} for a constant $C_3>0 $, using (\ref{eq:convergingzeros}). Moreover, we have  
\begin{eqnarray}
\int_0^{\infty} h_n^{\alpha}(x) \ud x & = & \int_0^{\infty} e^{- \left( 1 - \frac{\alpha}{2}\right)^2x^2 \sigma} x^{2n + \frac{\alpha}{2(2-\alpha)}} \ud x \nonumber \\
& = & \int_0^{\infty} e^{-t} \left( \frac{2}{2-\alpha} \sqrt{\frac{t}{\sigma}}\right)^{2n + \frac{\alpha}{2(2-\alpha)}} \frac{1}{2\sqrt{\sigma t}}\left( \frac{2}{2-\alpha} \right) \ud t \nonumber \\
& = & \frac{1}{2} \left[ \frac{1}{\sqrt{\sigma}} \left( \frac{2}{2-\alpha} \right) \right]^{2n + \frac{\alpha}{2(2-\alpha)} + 1} \int_0^{\infty}e^{-t}t^{n+ \frac{\alpha}{4(2-\alpha)} - \frac{1}{2}} \ud t \nonumber \\
& = & \frac{1}{2} \left[ \frac{1}{\sqrt{\sigma}} \left( \frac{2}{2-\alpha} \right) \right]^{2n + \frac{\alpha}{2(2-\alpha)} + 1} \Gamma \left( n + \frac{\alpha}{4(2-\alpha)} + \frac{1}{2} \right),   \nonumber
\end{eqnarray} where we have used (\ref{eq:Gammafunction}) with $p=n + \frac{\alpha}{4(2-\alpha)} + \frac{1}{2}$. Hence, combining this with (\ref{eq:decoupage}) and (\ref{eq:majorationpremierterme}), we get
\begin{equation*}
\sum_{k=1}^{\infty} h_n^{\alpha}(j_{\nu,k}) \leq \left( C_2 + \frac{C_3}{\sqrt{\sigma}(2-\alpha)} \right) \left[ \frac{1}{\sqrt{\sigma}} \left( \frac{2}{2-\alpha} \right) \right]^{2n + \frac{\alpha}{2(2-\alpha)}} \Gamma \left( n + \frac{\alpha}{4(2-\alpha)} + \frac{1}{2} \right),
\end{equation*} which, according to (\ref{eq:seriesBnk}), implies
\begin{equation*}
\sum_{k=1}^{\infty}|B_{n,k}(t,x)| \leq  C_4\left[ \frac{1}{\sqrt{\sigma}} \left( \frac{2}{2-\alpha} \right) \right]^{2n + \frac{\alpha}{2(2-\alpha)}} \frac{\Gamma \left( n + \frac{\alpha}{4(2-\alpha)} + \frac{1}{2}\right)}{2^{2n}n! \Gamma \left( n + \nu +1 \right)}.
\end{equation*} Henceforth, the D'Alembert criterium for entire series gives (\ref{eq:convergencevaleurabsolue}). \par
\bigskip

\item[Step 2] We find $Y \in \G^1([\sigma,T])$  such that (\ref{eq:trueexplicitsolution}) holds. \par 
Thanks to Fubini's theorem, (\ref{eq:Fubini}) and (\ref{eq:Gammafunction2}), we may write 
\begin{equation*}
f(t,x) = \sum_{n=0}^{\infty} \frac{ y_n(t) \left( x^{1 - \frac{\alpha}{2}} \right)^{2n}}{(2-\alpha)^{2n} n! \prod_{j=1}^n (j + \nu)},
\end{equation*} where, for every $n\in \N$, 
\begin{equation*}
y_n(t) := \frac{(-1)^n \sqrt{2-\alpha} \left(1 - \frac{\alpha}{2}  \right)^{2n}}{2^{\nu}\Gamma\left( \frac{1}{2-\alpha} \right)} \sum_{k=1}^{\infty}a_k e^{-\lambda_k t} \frac{j_{\nu,k}^{2n + \nu}}{|J_{\nu +1}(j_{\nu,k})|}, \quad \forall t \in [\sigma,T],
\end{equation*} and $\nu$ is given by (\ref{eq:nu}). Putting
\begin{equation}
Y(t) := \frac{\sqrt{2-\alpha}}{2^{\nu}\Gamma\left( \frac{1}{2-\alpha}\right)} \sum_{k=1}^{\infty} \frac{a_k j_{\nu,k}^{\nu}}{|J_{\nu +1}(j_{\nu,k})|}e^{-\left( 1 - \frac{\alpha}{2} \right)^2 j_{\nu,k}^2 t}, \quad t \in [\sigma, T],
\label{eq:flatoutput}
\end{equation} we have that, since $\sigma >0$, $Y$ is analytic in $[\sigma,T]$. Moreover, 
\begin{equation*}
Y^{(n)}(t) = y_n(t), \quad \forall t \in [\sigma,T] , \, \forall n\in \N.
\end{equation*} Hence, we obtain (\ref{eq:trueexplicitsolution}) with this choice. Since $\sigma \in (0,T)$ is arbitrary, we have in addition that $Y \in \C^{\infty}((0,T])$.\par
Furthermore, applying Proposition \ref{proposition:convergence} to (\ref{eq:trueexplicitsolution}) with $t_1=\sigma$ and $t_2 = T$, we deduce that $f$ solves (\ref{eq:Cauchypbm}) pointwisely in $(\sigma,T) \times (0,1)$ with $u =0$ and $f_{\sigma}(x) = f(\sigma,x)$.

\end{description}
\end{proof}

\section{Construction of the control}

Let $s\in \R$ with $s>1$. The function (see \cite[Section 2]{MRR} and \cite[Theorem 11.2, p.48]{Widder})
\begin{equation}
\phi_s(t) := \left\{ \begin{array}{ll}
1, & \textrm{ if } t \leq 0, \\
\frac{e^{-(1-t)^{-\frac{1}{s-1}}}}{e^{-(1-t)^{-\frac{1}{s-1}}} + e^{-t^{-\frac{1}{s-1}}}}, & \textrm{ if } 0< t <1, \\
0, & \textrm{ if } t \geq 1,
\end{array} \right.
\end{equation} belongs to $\G^s([0,1])$ and satisfies
\begin{equation}
\phi_s(0)=1,\, \phi_s(1) = 0, \quad \phi_s^{(i)}(0)=\phi_s^{(i)}(1) = 0,\, \forall i \in \N^*. 
\label{eq:propertiesofphis}
\end{equation}

\begin{proof}[Proof of Theorem \ref{thm:null-controllability}]
Let $f_0 \in L^2(0,1)$, $T>0$. Let $f$ and  $Y$ be given by Proposition \ref{proposition:regularisingeffect}.  \par 
We pick $\tau \in (0, T)$, $s \in (1,2)$ and we set the flat output
\begin{equation*}
y(t) := \phi_s \left( \frac{t - \tau}{T-\tau} \right) Y(t), \quad \forall t\in (0, T], 
\end{equation*} which belongs to $\C^{\infty}(0,T)$. Moreover, for every $\sigma \in (0,T)$,  $y \in \G^s([\sigma,T])$, as it is a product of two functions in $\G^s([\sigma,T])$. We define accordingly the function 
\begin{equation*}
\tilde{f}(t,x) := \sum_{k=1}^{\infty} \frac{y^{(n)}(t) \left( x^{1-\frac{\alpha}{2}} \right)^{2n}}{(2-\alpha)^{2n} n! \prod_{j=1}^n \left( j + \frac{\alpha -1}{2 -\alpha} \right)}, \quad \forall (t,x) \in (0, T] \times [0,1],
\end{equation*} and the control
\begin{equation}
u(t) = \left\{ \begin{array}{ll}
0, & t\in [0, \tau], \\
\tilde{f}(t,1), & t \in (\tau, T].
\end{array} \right.
\label{eq:vraicontrole2}
\end{equation} Since $y \in \G^s([\sigma, T])$ for some $s\in (1,2)$, Proposition \ref{proposition:convergence} shows that 
\begin{equation}
\begin{array}{cc}
\forall \sigma \in (0,T),\, \tilde{f} \textrm{ is the pointwise solution of (\ref{eq:pointwiseCauchypbm}) with } \\
 t_1 = \sigma, \, t_2 =T, \, f_{t_1} = f(\sigma,\cdot) \textrm{ and } (\ref{eq:vraicontrole2}).
\end{array}
\label{eq:ftildepointwise}
\end{equation} As a consequence of (\ref{eq:propertiesofphis}), we have
\begin{eqnarray}
y(t) &=& Y(t), \,  \forall t\in (0,\tau], \nonumber \\
y(T) &=& 0. \label{eq:yequalszeroatT}
\end{eqnarray} Whence, $\tilde{f}(t,x) = f(t,x),$ for every $(t,x) \in (0,\tau) \times (0,1)$. Thus, as $f \in \C^0([0,T];L^2(0,1))$, we deduce
\begin{align}
\tilde{f} \in \C^0([0,T];L^2(0,1)), \label{eq:ftildecontinuous} \\
\tilde{f}(0) = f_0 \quad \textrm{in } L^2(0,1). 
\label{eq:tildeinitial}
\end{align} We have to check that $\tilde{f}$ is the weak solution of system (\ref{eq:Cauchypbm}) on $(0,T)$. To do so, and according to Definition \ref{def:weak}, let $t'\in (0,T)$ and let $\psi$ satisfying (\ref{eq:regularitetest}) and (\ref{eq:boundarytest}). Then, by (\ref{eq:ftildepointwise}) and since a pointwise solution is a weak solution (see Remark \ref{remark:uniqueness}), we have, for every $\sigma >0$,
\begin{gather}
\int_{\sigma}^{t'} \int_0^1 \tilde{f}(t,x) \left(\partial_t \psi + \partial_x (x^{\alpha} \partial_x \psi ) \right)(t,x) \ud t \ud x \nonumber \\
= \int_0^1 \tilde{f}(t',x)\psi(t',x) \ud x - \int_0^1 \tilde{f}(\sigma,x) \psi(\sigma,x) \ud x + \int_{\sigma}^{t'} u(t) \left(x^{\alpha}\partial_x \psi\right)(t,1) \ud t.  \nonumber 
\end{gather} Then, from (\ref{eq:vraicontrole2}), (\ref{eq:ftildecontinuous}), (\ref{eq:tildeinitial}) and (\ref{eq:regularitetest}), taking $\sigma \rightarrow 0^+$, we get the conclusion. \par
Finally, by construction (\ref{eq:yequalszeroatT}) implies that $\tilde{f}(T,x) = 0$, for every $x \in (0,1)$.

\end{proof}

\subsection*{Acknowledgements}
I thank Karine Beauchard for suggesting me this problem and for many fruitful discussions.


\begin{appendix}

\section{Some properties of the Gamma function}
For any $p\in \R^+$, the Gamma function is defined (see \cite[6.1.1, p.254]{AS}) by 
\begin{equation}
\Gamma(p) := \int_0^{\infty} e^{-t}t^{p-1} \ud t,
\label{eq:Gammafunction}
\end{equation} which is a monotone increasing function on $(0,\infty)$. Furthermore, (see \cite[6.1.15, p.256]{AS})
\begin{equation}
\Gamma(x+1) = x \Gamma(x), \,\,\, \forall x \in (0,\infty).
\label{eq:Gammafunction2}
\end{equation} We have the following asymptotics of the Gamma function.

\begin{lemma}[\cite{AS}, 6.1.39 ]
Let $a\in \R^+$ and $b\in \R$. Then,
\begin{equation}
\Gamma(ax + b) \underset{x \rightarrow \infty}{\sim} \sqrt{2 \pi}e^{-ax}(ax)^{ax +b -\frac{1}{2}}. 
\end{equation}  
\label{lemma:Stirling}
\end{lemma}

We show an inequality used in Proposition \ref{proposition:convergence}.

\begin{lemma}
\begin{equation}
(n+k)! \leq 2^{k+n} n! k!, \quad \forall n,k \in \N.
\label{eq:factorials}
\end{equation}
\end{lemma}

\begin{proof}
Let us observe first that 
\begin{equation}
(2n)! \leq 2^{2n} n!^2, \quad \forall n \in \N.
\label{eq:doublefactorial}
\end{equation} This inequality follows by induction, since, for every $n\in \N$, 
\begin{eqnarray}
\left(2(n+1) \right)! &=& (2n)!(2n+1)(2n+2) \nonumber \\
&\leq & (2n)!2^2(n+1)^2 \leq 2^{2(n+1)}(n+1)!. \nonumber
\end{eqnarray} To show (\ref{eq:factorials}), we assume, w.l.o.g., that $n<k$. Then, using (\ref{eq:doublefactorial}),
\begin{eqnarray}
(n+k)! &= & (2n)! \prod_{j=1}^{k-n}(2n + j) \nonumber \\
& \leq & (2n)! 2^{k-n} \prod_{j=1}^{k-n}(n + j) \leq 2^{n+k}n!k!. \nonumber 
\end{eqnarray}
\end{proof}

\section{Some properties of Bessel functions}
Let $\nu \in \R$. The Bessel function of order $\nu$ of the first kind is (\cite[9.1.10, p.360]{AS}) 
\begin{equation}
J_{\nu}(z) := \sum_{n=0}^{\infty}\frac{(-1)^n}{n! \Gamma(n+\nu +1)} \left( \frac{z}{2} \right)^{2n +\nu}, \,\,\, \forall z\in [0,\infty). 
\label{eq:Besselseries}
\end{equation} We denote by $\left\{j_{\nu,n} \right\}_{n \in \N^*}$ the increasing sequence of zeros of $J_{\nu}$, which are real for any $\nu \geq 0$ and enjoy the following properties (see \cite[9.5.2, p.370]{AS} and  \cite[Proposition 7.8, p.135]{KomLor}).
\begin{eqnarray}
\nu < j_{\nu,n} < j_{\nu, n+1}, \, \forall n \in \N^*,  \label{eq:decreasingzeros} \\
j_{\nu, n +1} - j_{\nu,n} \rightarrow \pi, \, \textrm{ as } n \rightarrow \infty. \label{eq:convergingzeros} 
\end{eqnarray} We also have the integral formula (\cite[11.4.5, p.485]{AS})
\begin{equation}
\int_0^1 y J_{\nu}(j_{\nu,n}y) J_{\nu}(j_{\nu,m}y) \ud y = \frac{1}{2}|J_{\nu +1}(j_{\nu,n})|^2 \delta_{n,m}, \,\,\, \forall n,m \in \N^*. 
\label{eq:Besselorthogonality} 
\end{equation} This allows to show the following.  
\begin{lemma}\cite[p.40]{Hig}
Let $\nu \geq 0$. The family $\left\{ w_n \right\}_{n\in\N^*}$ defined by
\begin{equation*}
w_n(z) := \frac{\sqrt{2z}}{|J_{\nu+1}(j_{\nu,n})|}J_{\nu}(j_{\nu,n}z), \quad \forall z\in (0,1),
\end{equation*} is an orthonormal basis of $L^2(0,1)$. In particular, if $f \in L^2(0,1)$ and $d_n := \int_0^1 f(z) w_n(z) \ud z, \, \forall n \in \N^* $, then $ \|f\|_{L^2(0,1)}^2 = \sum_{n=1}^{\infty}|d_n|^2.$ 
\label{lemma:Besselbase}
\end{lemma} We recall that $\forall \nu \in \R$, the Bessel function $J_{\nu}$ satisfies the following differential equation (see \cite[9.1.1, p.358]{AS})
\begin{equation}
z^2 J_{\nu}''(z) + z J_{\nu}'(z) + (z^2 - \nu^2)J_{\nu}(z) =0, \,\,\, \forall z\in(0,+\infty),
\label{eq:Besseldifferentialequation}
\end{equation} and the recurrence relation (see \cite[9.1.27, p.361]{AS}),
\begin{equation}
2J_{\nu}'(z) = J_{\nu-1}(z) + J_{\nu+1}(z), \,\,\, \forall z \in (0,+\infty).
\label{eq:Besselrecurrencerelation}
\end{equation} \par 

\subsection*{Asymptotic behaviour} We recall the asymptotic behaviour of $J_{\nu}$ for large arguments and near zero.

\begin{lemma}\cite[Lemma 7.2, p.129]{KomLor}
For any $\nu \in \R$, 
\begin{equation*}
J_{\nu}(z) = \sqrt{\frac{2}{\pi z}} \cos \left( z - \frac{\nu \pi}{2} - \frac{\pi}{4} \right) + \underset{z \rightarrow \infty}{O} \left( \frac{1}{z\sqrt{z}} \right).
\end{equation*} 
\label{lemma:asymptoticBessel}
\end{lemma} 

\begin{lemma}\cite[9.1.7, p.360]{AS}
For any $\nu \in \R \setminus \left\{-\N^*\right\}$, 
\begin{equation*}
J_{\nu}(z) \underset{z \rightarrow 0}{\sim} \frac{z^{\nu}}{2^{\nu}\Gamma(\nu+1)}.
\end{equation*}
\label{lemma:asymptoticBesselatzero}
\end{lemma}

The following asymptotic result is important in the proof of Proposition \ref{proposition:regularisingeffect}. We give the proof for the sake of completeness.
\begin{lemma}
Let $\nu \in \R^+$. Then,
\begin{equation}
\sqrt{j_{\nu,k}} |J_{\nu +1} (j_{\nu,k})| = \sqrt{\frac{2}{\pi}} + \underset{k\rightarrow \infty}{O}\left( \frac{1}{j_{\nu,k}} \right).
\label{eq:lemmaasymptoticBesselmajoration}
\end{equation} In particular, there exists a constant $C_1 >0$ such that for all $k \in \N^*$,
\begin{equation*}
\frac{1}{|J_{\nu +1} (j_{\nu,k})|} \leq C_1 \sqrt{j_{\nu,k}}.
\end{equation*}
\label{lemma:asymptoticBesselaux}
\end{lemma}
\begin{proof}
Using Lemma \ref{lemma:asymptoticBessel}, for $\nu +1$ and $x = j_{\nu,k}$,
\begin{eqnarray}
\sqrt{j_{\nu,k}}|J_{\nu +1}(j_{\nu,k})| &=& \sqrt{\frac{2}{\pi}} \left| \cos \left( j_{\nu,k} - \frac{\pi(\nu +1)}{2} - \frac{\pi}{4}  \right) \right| + \underset{k\rightarrow \infty}{O}\left( \frac{1}{j_{\nu,k}} \right) \nonumber \\
& = & \sqrt{\frac{2}{\pi}} \left| \sin \left( j_{\nu,k} - \frac{\pi \nu}{2} - \frac{\pi}{4}  \right)\right| + \underset{k \rightarrow \infty}{O}\left( \frac{1}{j_{\nu,k}} \right). \nonumber 
\end{eqnarray} Using again Lemma \ref{lemma:asymptoticBessel} with $\nu$ and $x = j_{\nu,k}$, we have that
\begin{equation*}
\cos \left( j_{\nu,k} - \frac{\pi \nu}{2} - \frac{\pi}{4}  \right) = \underset{k \rightarrow \infty}{O}\left( \frac{1}{j_{\nu,k}} \right),
\end{equation*} which gives
\begin{equation*}
\left|\sin \left( j_{\nu,k} - \frac{\pi \nu}{2} - \frac{\pi}{4}  \right) \right| = \sqrt{1 + \underset{k \rightarrow \infty}{O}\left( \frac{1}{j_{\nu,k}^2}\right)} = 1 + \underset{k \rightarrow \infty}{O} \left( \frac{1}{j_{\nu,k}} \right)
\end{equation*} and then (\ref{eq:lemmaasymptoticBesselmajoration}).

\end{proof}

\end{appendix}

\addcontentsline{toc}{section}{Bibliography}       
\bibliographystyle{plain}                            

\begin{thebibliography}{}

\end{thebibliography}


\begin{thebibliography}{99}
\bibitem{AS} M. Abramowitz, I. Stegun. Handbook of mathematical functions with formulas, graphs, and mathematical tables. \emph{ National Bureau of Standards. App. Math. series.} vol. 55. 1964.

\bibitem{ACF} F. Alabau-Boussourira, P. Cannarsa, G. Fragnelli. Carleman estimates for degenerate parabolic operators with applications to null controllability. \emph{ J. Evol. Equ.} vol. 6:2, pp. 161-204. 2006.

\bibitem{Allaire} G. Allaire. Introduction to Numerical Analysis and Optimization. \emph{Oxford Univ. Press.} 2007. 

\bibitem{CMP} M. Campiti, G. Metafune, D. Pallara. Degenerate self-adjoint evolution equations on the unit interval. \emph{Semigroup Forum.} vol. 57. pp. 1-36. 1998.

\bibitem{CFR} P. Cannarsa, G. Fragnelli, D. Rocchetti. Controllability results for a class of one-dimensional degenerate parabolic problems in nondivergence form. \emph{Discr. and Cont. Din. Syst. Series S.} vol. 6:3, pp. 687-701. 2013.

\bibitem{CMV4} P. Cannarsa, P. Martinez, J. Vancostenoble. Persistent regional null controllability for a class of degenerate parabolic equations. \emph{Comm. in Pure and Applied Anal.} vol. 3:4, pp. 607-635.  

\bibitem{CMV5} P. Cannarsa, P. Martinez, J. Vancostenoble. Null controllability of degenerate heat equations. \emph{Adv. in Diff. Eq.} vol. 10:2, pp.153-190. 2005.

\bibitem{CMV} P. Cannarsa, P. Martinez, J. Vancostenoble. Carleman estimates for a class of degenerated parabolic operators. \emph{SIAM J. Control Optim.} vol. 47:1, pp.1-19, 2008.

\bibitem{CMV2} P. Cannarsa, P. Martinez, J. Vancostenoble. Carleman estimates and null controllability for boundary-degenerate parabolic operators. \emph{C. R. Acad. Sci. S\'{e}r. I Math}. vol. 347, pp. 147-152. 2009.

\bibitem{CTY} P. Cannarsa, J. Tort, M. Yamamoto. Unique continuation and approximate controllability for a degenerate parabolic equation. \emph{Applicable Analysis.} vol. 91:8, pp. 1409-1425. 2012.

\bibitem{Coron} Coron, Jean-Michel. Control and nonlinearity. Mathematical Surveys and Monographs, 136. \emph{American Mathematical Society}, Providence, RI, 2007.

\bibitem{Fliess}
M. Fliess, J. L. L\'{e}vine, P. Martin, P. Rouchon. Flatness and defect of non-linear systems: introductory theory and examples. \emph{Int. J. of Control}. vol. 61:6, pp. 1327-1361, 1995. 

\bibitem{Gueye} M. Gueye. Exact boundary controllability of 1-D parabolic and hyperbolic degenerate equations. \emph{SIAM J. Control Opt.} 52:4, pp. 2037-2054. 2014.

\bibitem{Hig} J. R. Higgins. Completeness and basis properties of sets of special functions. \emph{Cambridge Univ. Press. Cambridge Tracts in Mathematics,} Vol. 72. 1978.

\bibitem{KomLor} V. Komornik, P. Loreti. Fourier series in Control Theory. \emph{Springer Monographs in Mathematics}. 2005.

\bibitem{LMR} B. Laroche, P. Martin, P. Rouchon. Motion planning for the heat equation. \emph{Int. J. Robust Nonlinear Control,}  10:8, pp. 629-643, 2000.

\bibitem{MRR} P. Martin, L. Rosier, P. Rouchon. Null controllability of the heat equation using flatness. \emph{Automatica}, vol. 50, pp. 3067-3076. 2014.

\bibitem{MRR2} P. Martin, L. Rosier, P. Rouchon. Null controllability of one-dimensional parabolic systems using flatness. \emph{arXiv:1410.2588}. 2014.

\bibitem{MRR3} P. Martin, L. Rosier, P. Rouchon. Controllability of the 1D Schrodinger equation by the flatness method. \emph{arXiv:1404.0814}. 2014.

\bibitem{V} J. Vancostenoble. Improved Hardy-Poincaré inequalities and sharp Carleman estimates for degenerate/singular parabolic problems. \emph{Discr. and Cont. Din. Syst. Series S.}  vol. 4:3, pp. 761-790. 2011.

\bibitem{Widder} D.V. Widder. The Heat Equation. \emph{Academic Press. Pure and Applied Mathematics series,} Vol. 67. 1975. 

\end{thebibliography}

\end{document}